\documentclass[10pt,a4paper]{article}
\usepackage[T1]{fontenc}
\usepackage[utf8]{inputenc}
\usepackage{authblk}
\usepackage{amsmath,amssymb,amsthm,esint,bm}
\usepackage{mathrsfs}
\usepackage{bookmark}
\usepackage{amsmath}
\usepackage{enumerate}
\allowdisplaybreaks[3]

\newcounter{stepnum}
\newcommand{\step}{%
	\par
	\refstepcounter{stepnum}%
	\textbf{Step \arabic{stepnum}}.\enspace\ignorespaces
}

\newtheorem{definition}{Definition}[section]
\newtheorem{theorem}[definition]{Theorem}
\newtheorem{lemma}[definition]{Lemma}
\newtheorem{proposition}[definition]{Proposition}
\newtheorem{corollary}[definition]{Corollary}
\theoremstyle{remark}
\newtheorem{remark}[definition]{Remark}
\numberwithin{equation}{section}

\title{Global Calder\'{o}n-Zygmund estimates for asymptotically convex fully nonlinear Grad-Mercier type equations}
\author[a]{Yao Zhang}
\author[a]{Xiaofeng Jin}
\author[a]{Lingwei Ma}
\author[a]{Zhenqiu Zhang\thanks{Corresponding author.}}

\affil[a]{School of Mathematical Sciences and LPMC, Nankai University, Tianjin, 300071, P.~R.~China}

\date{\today}

\usepackage{hyperref}
\begin{document}
\maketitle
\footnotetext[1]{E-mail: 2120220099@mail.nankai.edu.cn (Y. Zhang), 1120220040@mail.nankai.edu.cn (X. Jin), malingwei@nankai.edu.cn (L. Ma), zqzhang@nankai.edu.cn (Z. Zhang).}
\begin{abstract}
In this paper, we consider the following Dirichlet problem for the fully nonlinear elliptic equation of Grad-Mercier type under asymptotic convexity conditions
\begin{equation*}
	\left\{
	\begin{array}{ll}
		F(D^2u(x),Du(x),u(x),x)=g(|\{y\in \Omega:u(y)\ge u(x)\}|)+f(x) & \text{in } \Omega, \\
		u=\psi &\text{on } \partial \Omega.
	\end{array}
	\right.
\end{equation*}
In order to overcome the non-convexity of the operator $F$ and the nonlocality of the nonhomogeneous term $g$, we apply the compactness methods and frozen technique to prove the existence of the $W^{2,p}$-viscosity solutions and the global $W^{2,p}$ estimate. As an application, we derive a Cordes-Nirenberg type continuous estimate up to boundary. Furthermore, we establish a global BMO estimate for the second derivatives of solutions by using  an asymptotic approach, thereby refining the borderline case of Calder\'{o}n-Zygmund estimates.



\end{abstract}


\section{Introduction and main results}\label{section1}

In this paper, we are concerned with existence and the global Calder\'{o}n-Zygmund regularity for $W^{2,p}$-viscosity solutions of the following fully nonlinear Dirichlet problem
\begin{equation}\label{eqn}
	\left\{
	\begin{array}{ll}
		F(D^2u(x),Du(x),u(x),x)=g(|\{y\in \Omega:u(y)\ge u(x)\}|)+f(x) &\text{in } \Omega, \\
		u=\psi &\text{on } \partial \Omega,
	\end{array}
	\right.
\end{equation}
where $ \Omega\subseteq \mathbb{R}^d$ is a bounded domain with $C^{1,1}$ boundary, and throughout this paper $g:[0,|\Omega|]\to \mathbb{R}$ is a continuous function. The fully nonlinear elliptic operator $F:\mathcal{S}(d)\times \mathbb{R} ^{d} \times\mathbb{R}\times\Omega \to \mathbb{R}$  is uniformly elliptic with ellipticity constants $0<\lambda\le\Lambda$. We assume that F satisfies  some structure conditions, i.e., there exist universal constants $\gamma\geq 0$ and $\omega\geq 0$ such that
\begin{equation}\label{F}
	\mathcal{M}^{-}(M-N)-\gamma|p-q|-\omega|r-s|\leq F(M, p, r, x)-F(N, q, s, x)\leq \mathcal{M}^{+}(M-N)+\gamma|p-q|+\omega|r-s|.\tag{SC}
\end{equation}
Here $\mathcal{S}(d)$ denotes the set of $d\times d$ real symmetric matrices, while $\mathcal{M}^{+}$ and $\mathcal{M}^{-}$ represent Pucci's extremal operators defined by
\begin{equation*}
	\begin{aligned}
		\mathcal{M}^{-}(M, \lambda, \Lambda)&=\mathcal{M}^{-}(M)=\lambda \sum_{e_{i}>0} e_{i}+\Lambda \sum_{e_{i}<0} e_{i}, \\
		\mathcal{M}^{+}(M, \lambda, \Lambda)&=\mathcal{M}^{+}(M)=\Lambda \sum_{e_{i}>0} e_{i}+\lambda \sum_{e_{i}<0} e_{i},
	\end{aligned}
\end{equation*}
where  $\{e_{i}\}_{i\ge 1}$  are the eigenvalues of  $M$. Some basic properties of Pucci's extremal operators are listed in \cite{caffarelli1995fully}.

The Grad-Mercier type equations  of the form \eqref{eqn}, introduced by Grad \cite{grad1,grad2,grad3}  and Mercier \cite{mercier}, simplify the original fully nonlinear operator by replacing it with a Laplace operator. These equations are also referred to as queer differential equations (QDEs).	
This type of equation is commonly used to characterize many physical problems in the field of adiabatic compression and resistive diffusion of plasma in toroidal or spherical containers. Recently, there has been extensive research on the qualitative properties of solutions to equations about plasma physics, which can be referenced in \cite{1978,1979,1981,1982}.

The celebrated contributions about the qualitative properties of solutions were established by Caffarelli \cite{1989} (see also \cite{1996}) for the following fully nonlinear elliptic equations
\begin{equation}\label{fully}
	F(D^2u,x)=f(x)\quad\text{in }\Omega,
\end{equation}
where $F$ is uniformly elliptic and $f\in L^{p}(\Omega)$ with $p>d$. He established the interior  $W^{2,p}$ estimates for $L^p$-viscosity solutions where he assumed the homogeneous equation driven by the frozen coefficients counterpart of the operator to have $C^{1,1}$ estimates. Furthermore, under the same assumptions as the main result in \cite{1989}, the interior $W^{2,p}$ regularity have been generalized by Escauriaza \cite{escauriaza1993w} for the solution of \eqref{fully} whose integral exponent of the source term allows that $p > d-\epsilon$ for some $\epsilon > 0$ depending on the ellipticity constants and the spatial dimension. At the same time, many authors were inspired to research fully nonlinear equations with low order terms
\begin{equation}\label{fully2}
	F(D^2u,Du,u,x)=f(x)\quad\text{in }\Omega,
\end{equation}
where $F$ satisfies structure condition \eqref{F} and $f\in L^{p}(\Omega)$ with $p>d$.
Under the assumption of the smallness of oscillation measure in $L^p$ spaces, \'{S}wiech \cite{du1} obtained the existence and interior $W^{2,p}$ estimates for $L^p$-viscosity solution  to \eqref{fully2}, where $p>p_0=p_0(d,diam(\Omega),\lambda,\Lambda,\gamma)$. Winter \cite{winter} extended the results of Caffarelli and \'{S}wiech. He proved $W^{2,p}$ and $W^{1,p}$ estimates at a flat  boundary. Besides, he also established the existence and uniqueness of solutions for the equation \eqref{fully2} with Dirichlet boundary condition.
For concave or convex fully nonlinear elliptic equations, there have been a quantity of literature on  Sobolev regularity of $W^{2,p}$-viscosity solutions. We may refer to \cite{evans1,caffarelli1995fully,armstrong,2003,krylov1983}.

The challenging question for fully nonlinear elliptic equations with nonconvex structures is whether Sobolev estimates are valid. Without considering the convexity of the operator and under the condition $p\le d$, \'{S}wiech \cite{du1} obtained that $L^p$ viscosity solutions of equation \eqref{fully2} satisfied $W^{1,q}$ regularity, where
\begin{equation*}
	q<p^{\ast}=\frac{dp}{d-p}.
\end{equation*}
In \cite{krylov1,krylov2}, Krylov established the existence of $W^{2,p}$-viscosity solutions of the Dirichlet problem for fully nonlinear elliptic equations under relaxed convexity and no convexity assumptions, respectively. Furthermore , Pimentel and  Teixeira \cite{Teixeira2016} introduced the following concept of asymptotic convexity. More precisely, for  $\mu>0$, they set	$F_{\mu}(M)=\mu F(\mu^{-1} M),$ and the recession function $F^{\ast}$ associated with $F$  is  defined by
\begin{equation*}
	F^{\ast}(M):=\lim _{\mu \rightarrow 0} F_{\mu}(M).
\end{equation*}
Under the assumption that $F^{\ast}$ is convex, they studied the interior $W^{2,p}$ estimates for solutions of
the following fully nonlinear elliptic equation
\begin{equation}\label{fully3}
	F(D^{2}u)=f \quad\text{in } \Omega
\end{equation}
with $f\in L^{p}(\Omega), p>d$.

Inspired by the definition of asymptotic convexity for operators $F$ without lower-order terms, some researchers have defined the concept of asymptotic convexity for operators that include lower-order terms.
The authors derived the interior $W^{2,p}$ regularity of solutions for \eqref{fully2} while the assumption of the growth rate of the gradient of the solution in the original structural condition is relaxed by an $L^p$ integrable function in \cite{haozhang} under the condition that operators $F$ with lower-order terms have asymptotic convexity.
It is worth pointing out that some global $W^{2,p}$ regularity results to the equation \eqref{fully2} with the boundary value function $\psi\in W^{2,p}(\Omega)$ were proved by da Silva and Ricarte \cite{bmo} while they relaxed the convexity of $F^{\ast}$ to the extent that the solution of the homogeneous equation corresponding to $F^{\ast}$ satisfies the $C^{1,1}$ \textit{a priori} estimate (see also Definition \ref{F^ast}).

Recently, Caffarelli and Tomasetti \cite{plasma} considered the convex fully nonlinear elliptic operator $F$ and have studied the existence and global $W^{2,p}$ estimates for $W^{2,p}$-viscosity solutions to the following Grad-Mercier type Dirichlet problem,
\begin{equation}\label{fully4}
	\left\{
	\begin{array}{ll}
		F(D^2u(x))=g(|\{y\in \Omega:u(y)\ge u(x)\}|) & \text{in } \Omega, \\
		u=\psi &\text{on } \partial \Omega,
	\end{array}
	\right.
\end{equation}
where $\psi\in W^{2,p}(\Omega)$ with $p>d$ and $g:[0,|\Omega|]\to \mathbb{R}$ is a continuous function. If the operator $F$ lacks convexity, the study of the Dirichlet problem derived from the fully nonlinear Grad-Mercier type equation remains insufficient. The core challenges of this problem stem from the geometric degeneracy caused by the non-convexity of the operator and the non-local effects introduced by the inhomogeneous term on the right-hand side of the equation.

To address these issues, this paper aims to investigate  the Dirichlet boundary value problem for the fully nonlinear Grad-Mercier type equation with lower-order terms under asymptotic convexity conditions. After resolving the above challenges, we prove the existence of $W^{2,p}$-viscosity solutions and the global Calder\'{o}n-Zygmund estimates for the Dirichlet boundary value problem of  the  asymptotically convex fully nonlinear Grad-Mercier type equations. Before presenting  the main results of this paper, we need some notations.
At first,  the recession operator is rigorously defined as follows.
\begin{definition}
	For $\mu>0$, we set
	\begin{equation*}
		F_{\mu}(M,0,0,x)=\mu F(\mu^{-1} M,0,0,x).
	\end{equation*}
	The recession function associated with $F$  is  defined by
	\begin{equation*}
		F^{*}(M,0,0,x):=\lim _{\mu \rightarrow 0} F_{\mu}(M,0,0,x) .
	\end{equation*}
\end{definition}
Next we define the smallness condition on the oscillation measure for $F^{\ast}$.

\begin{definition}
	For $x_{0} \in \Omega$, the oscillation measure of $F^{\ast}$ around  $x_{0}$ defined as
	\begin{equation*}
		\beta_{F^{*}}\left(x, x_{0}\right):=\sup _{X \in \mathcal{S}(d)} \frac{\left|F^{*}(X, 0,0, x)-F^{*}\left(X, 0,0, x_{0}\right)\right|}{\|X\|+1}.
	\end{equation*}
	We say $F^{\ast}$ satisfies the smallness condition on the oscillation measure in $L^p$ if there exist universal constants  $ \alpha_0 \in(0,1)$, $\theta_{0}>0$  and  $0<r_{0} \ll 1 $ such that for $ 0<r \ll r_{0}$ and $p>d$,
	\begin{equation}\label{beta}
		\left(\fint_{B_{r}\left(x_{0}\right) \cap \Omega}|\beta_{F^{*}}\left(x, x_{0}\right)|^{p} d x\right)^{\frac{1}{p}} \leq \theta_{0} r^{\alpha_0},
	\end{equation}
	for every $x_0\in\Omega$.
\end{definition}

Now we are in a position to present the main results of this paper.  We first  state the existence result  and  the global $W^{2,p}$ regularity of $W^{2,p}$-viscosity solutions to Dirichlet problem \eqref{eqn}.
\begin{theorem}\label{main}
	Let  $F$  satisfy \eqref{F}. Suppose that the recession operator $F^\ast$ is convex and satisfies the smallness condition on the oscillation measure in $ L^p$ with $p>d$ . If $\psi\in W^{2,p}(\Omega)$ and $f\in L^{p}(\Omega)$, then there exists a $W^{2,p}$-viscosity solution $u$ to \eqref{eqn}.
	Moreover, $u \in W^{2, p}\left(\Omega\right)$  and there exists universal constant $C=C(d,p,\lambda,\Lambda,\theta_0,\alpha_0,\gamma,diam(\Omega),\|\partial\Omega\|_{C^{1,1}})>0$  such that
	\begin{equation*}
		\|u\|_{W^{2, p}\left(\Omega\right)} \leq C\left(\|u\|_{L^{\infty}\left(\Omega\right)}+\|\psi\|_{W^{2,p}\left(\Omega\right)}+\|g(|u\ge u(x)|)\|_{L^p(\Omega)}+\|f\|_{L^p(\Omega)}\right).
	\end{equation*}
\end{theorem}
As a consequence of Theorem \ref{main}, it is obvious to derive the H\"{o}lder regularity for gradient of solution by applying the Sobolev embedding theorem. Meanwhile, the global Cordes-Nirenberg type estimate could be proved.
\begin{corollary}\label{c1}
	Under the same assumptions of Theorem \ref{main}, we obtain that
	$u\in C^{1,\alpha}(\bar{\Omega})$ for any $0<\alpha<1$. Futhermore,
	there exists universal constant $C=C(d,p,\lambda,\Lambda,\theta_0,\alpha_0,\gamma,diam(\Omega),$ $\|\partial\Omega\|_{C^{1,1}})>0$  so that
	\begin{equation*}
		\|u\|_{C^{1, \alpha}\left(\bar{\Omega}\right)} \leq C\left(\|u\|_{L^{\infty}\left(\Omega\right)}+\|\psi\|_{W^{2,p}\left(\Omega\right)}+\|g(|u\ge u(x)|)\|_{L^p(\Omega)}+\|f\|_{L^p(\Omega)}\right).
	\end{equation*}
\end{corollary}
\begin{remark}
	There exist many H\"{o}lder continuity results for fully nonlinear elliptic equations. Caffarelli \cite{1989,caffarelli1995fully} obtained the
	interior $C^{1,\alpha}$ estimates to the equation \eqref{fully}. Considering Grad-Mercier type equation, Caffarelli and Tomasetti \cite{plasma} proved the $C^{1,\alpha}$ regularity for \eqref{fully4} in the case that $F$ has convexity. Compared with \cite{plasma}, we obtain $C^{1,\alpha}$ estimates from Corollary \ref{c1} for the Grad-Mercier type Dirichlet problem \eqref{eqn} under the assumption that recession operator $F^{\ast}$ has convexity.
\end{remark}
Finally, we turn to the borderline situation. We introduce the definition of $p\operatorname{-BMO}$ space.
\begin{definition}
	\begin{enumerate}[(i)]
		\item 	Assume $u\in p\operatorname{-BMO}(\Omega)$ for $p\ge 1$ if $u\in L^p_{loc}(\Omega)$ and for $x_0\in \Omega$, $\rho>0$, define the $p\operatorname{-BMO}$ seminorm
		\begin{equation*}
			[u]_{p-BMO(\Omega)}:=\sup_{B_\rho\subset\Omega} \left(\fint_{B}|u(x)-(u)_{x_0,\rho}|^{p} d x\right)^{\frac{1}{p}} < \infty,
		\end{equation*}
		where $(u)_{x_0,\rho}=\fint_{B_\rho(x_0)\cap\Omega}u(x)dx$. When $p=1$ one says simply that $u\in \operatorname{BMO}(\Omega)$.
		\item Furthermore, define the $p\operatorname{-BMO}$ norm as follows
		\begin{equation*}
			\left \| u \right \|_{p\operatorname{-BMO}(\Omega)}:=[u]_{p\operatorname{-BMO}(\Omega)}+\left \| u \right \|_{L^p(\Omega)}.
		\end{equation*}
	\end{enumerate}
\end{definition}
\begin{remark}
	As a consequence of John-Nirenberg inequality, there is an obvious fact that the seminorms of $p\operatorname{-BMO}$ space
	and $\operatorname{BMO}$ space are equivalent for all $p\in[1,+\infty)$.
\end{remark}
When $f\in p\operatorname{-BMO}(\Omega)$, we establish that $W^{2,p}$-viscosity solutions admit $\operatorname{BMO}$ type estimates for their second derivatives. The complete Calder\'{o}n-Zygmund estimates for $W^{2,p}$-viscosity solutions to \eqref{eqn} are derived.
\begin{theorem}\label{pbmo}
	Let  $F$  satisfy \eqref{F}. Suppose that the recession operator $F^\ast$ is convex and satisfies the smallness condition on the oscillation measure in $ L^p$ with $p>d$ . If $D^2\psi\in p\operatorname{-BMO}(\Omega)$ and $f\in p\operatorname{-BMO}(\Omega)$, then there exists a $W^{2,p}$-viscosity solution $u$ to problem \eqref{eqn} such that
	\begin{equation*}
		\|D^2u\|_{p\operatorname{-BMO}(\Omega)}\le C(\left \| u \right \|_{L^\infty(\Omega)}+\|D^2\psi\|_{p\operatorname{-BMO}(\Omega)}+\|g(|u\ge u(x)|)\|_{p\operatorname{-BMO}(\Omega)}+\|f\|_{p\operatorname{-BMO}(\Omega)}),
	\end{equation*}
	where $C=C(d,p,\lambda,\Lambda,\theta_0,\alpha_0,\gamma,\omega,diam(\Omega),\|\partial\Omega\|_{C^{1,1}})>0$.
\end{theorem}

The remainder of this paper is organized as follows. Section 2 contains some notions and preliminary conclusions. In section 3, we establish the existence of solutions and prove the $W^{2,p}$ estimates for Dirichlet problem  associated with  the fully nonlinear elliptic equation of Grad-Mercier type.  Finally, section 4 is devoted to the proof of BMO type estimates for $W^{2,p}$-viscosity solutions to \eqref{eqn}.
\section{Some Preliminaries}\label{section2}
In this section, we introduce some definitions and classical results which will be used in this article. First, we present the definition of $W^{2,p}$-viscosity solutions.

\begin{definition}
	Let $F:\mathcal{S}(d)\times\mathbb{R}^{d}\times\mathbb{R}\times\Omega \backslash \mathcal{N}\rightarrow \mathbb{R}$ be a measured and uniformly elliptic operator where $\mathcal{N}\subset\Omega$ be a null set,
	with respect to the Lebesgue measure. Suppose $F$ does not increase with respect to $M$ and let $f\in L^{p}(\Omega)$, $p>d/2$.
	\begin{enumerate}[(i)]
		\item A function  $u$ is upper semi-continuousis in $\Omega$ and is a $W^{2,p}$-viscosity subsolution to
		\begin{equation}\label{2.1}
			F\left(D^{2} u,D u,u,x\right)=f(x)\quad in \quad\Omega,
		\end{equation}
		if $u\le\psi$ on $\partial\Omega$, and for every  $\varphi \in W_{\text {loc }}^{2, p}(\Omega) $ such that there exist  $\varepsilon>0 $ and an open set $ U \subset \Omega$  for which
		\begin{equation*}
			F\left(D^{2} \varphi(x), D \varphi(x),u(x),x\right)-f(x) \leq -\varepsilon
		\end{equation*}
		a.e. $x \in U$, then  $u-\varphi$  does not attain a local maximum in $ U$.
		\item A function  $u$ is lower semi-continuous in $\Omega$ and is a $ W^{2,p}$-viscosity supersolution to \eqref{2.1} if $u\ge\psi$ on $\partial\Omega$ and for every $\varphi\in W_{\text{loc}}^{2, p}(\Omega)$  such that there exist $\varepsilon>0$ and an open set  $U \subset \Omega$  for which
		\begin{equation*}
			F\left(D^{2} \varphi(x), D \varphi(x),u(x),x\right)-f(x) \geq \varepsilon
		\end{equation*}
		a.e. $ x \in U$, then  $u-\varphi$  does not attain a local minimum in $ U$.
		\item If  $u \in C(\Omega)$  is simultaneously a $ W^{2,p}$-viscosity subsolution and a supersolution to \eqref{2.1}, it is called that  $u$  is an $ W^{2,p}$-viscosity solution to \eqref{2.1}.
	\end{enumerate}
\end{definition}

Now we introduce the classification of viscosity solutions through the partition of the fully nonlinear elliptic equations formed by Pucci's extremal operator as follows.

\begin{definition}
	Let $f \in L^{p}(\Omega)$  with   $p\ge d$, $0<\lambda \leq \Lambda$ and $\gamma\ge 0$. Define that
	\begin{enumerate}[(i)]
		\item 	$\underline{S}(\lambda, \Lambda,\gamma, f)$ is the class of functions $u \in C(\Omega)$ satisfying
		\begin{equation*}
			\mathcal{M}^{+}\left(D^{2} u, \lambda, \Lambda\right)+\gamma|Du| \geq f(x)
		\end{equation*}	
		in the viscosity sense in $\Omega$.
		\item $\bar{S}(\lambda, \Lambda, \gamma,f)$ is the class of functions  $u \in C(\Omega)$ satisfying
		\begin{equation*}
			\mathcal{M}^{-}\left(D^{2} u, \lambda, \Lambda\right)-\gamma|Du| \leq f(x)
		\end{equation*}
		in the viscosity sense in $\Omega$.
		\item Furthermore,\begin{equation*}
			\begin{aligned}
				S(\lambda, \Lambda,\gamma, f)  &=\underline{S}(\lambda, \Lambda,\gamma, f) \cap \bar{S}(\lambda, \Lambda, \gamma, f), \\
				S^{*}(\lambda, \Lambda, \gamma, f)  &=\underline{S}(\lambda, \Lambda, \gamma,-|f|) \cap \bar{S}(\lambda, \Lambda, \gamma,|f|) .
			\end{aligned}
		\end{equation*}
	\end{enumerate}
	
\end{definition}

In the same setting as \cite{1996}, we could define the $W^{2,p}$-strong solution.

\begin{definition}
	Define that $u$ is a $W^{2,p}$-strong subsolution of \eqref{2.1} in $\Omega$, if $u\in W^{2,p}_{loc}(\Omega)$, $u\le\psi$ on $\partial\Omega$ and
	\begin{equation*}
		F\left(D^2u(x),Du(x),u(x),x\right)\ge f(x)\text{ a.e. in }\Omega.
	\end{equation*}
\end{definition}
Additionally, we need to use the following key results in the sequent proof. Firstly, we give the Alexandroff-Bakel'man-Pucci maximal principle.
\begin{proposition}\label{ABP}\cite[Proposition 2.12]{1996}
	Let $f\in L^p(\Omega)$, $p>d$. If $u\in C(\bar{\Omega})$ and $u\in\bar{S}(\lambda, \Lambda, \gamma,f)$ in $\{u>0\}$, then
	\begin{equation*}
		\sup_{\Omega}u\le\sup_{\partial\Omega}u^{+}+ C(\lambda,\Lambda,\gamma,d,p,diam(\Omega))\|f^{+}\|_{L^p(\Gamma^{+}(u^{+}))}.
	\end{equation*}
\end{proposition}

The following theorem concerns the stability of $W^{2,p}$-viscosity solutions.
\begin{proposition}\label{stability}\cite[Theorem 1.17]{MR442005820220101}
	Let $\left\{F_{n}\right\}_{n \in \mathbb{N}},  \left\{f_{n}\right\}_{n \in \mathbb{N}} \subset L^{p}(\Omega) $, $p>d$ and  $\left\{u_{n}\right\}_{n \in \mathbb{N}} \subset C(\Omega)$  be sequences such that,
	\begin{enumerate}[(i)]
		\item For every  $n \in \mathbb{N}$, the operator  $F_{n}:\mathcal S(d)\times\mathbb{R}^{d}\times\mathbb{R}\times\Omega \backslash \mathcal{N} \rightarrow \mathbb{R}$  satisfies \eqref{F}.
		\item For every $n \in \mathbb{N}$, the function $u_{n}$ is a $W^{2,p}$-viscosity solution to
		\begin{equation*}
			F_{n}\left(D^{2} u_{n},D u_{n},u_n,x\right)=f_{n} \quad \text { in } \quad \Omega.
		\end{equation*}
	\end{enumerate}
	Suppose there exists $u_{\infty} \in C(\Omega)$  such that $u_{n} \rightarrow u_{\infty}$ locally uniformly as $n \rightarrow \infty$. Suppose further there are  $F_{\infty}$ and  $f_{\infty}$ such that, for every  $B_{r}\left(x_{0}\right) \subset \Omega$  and  $\phi \in W^{2, p}\left(B_{r}\left(x_{0}\right)\right)$ the function
	\begin{equation*}
		\begin{aligned}
			g_{n}(x):= & F_{n}\left( D^{2} \phi(x),D \phi(x),u_n(x),x\right)-F_{\infty}\left( D^{2}\phi(x),D\phi(x),u_{\infty}(x),x\right) \\
			& +f_{\infty}(x)-f_{n}(x)
		\end{aligned}
	\end{equation*}
	converges to zero in  $L^{p}\left(B_{r}\left(x_{0}\right)\right)$ as  $n \rightarrow \infty$. Then  $u_{\infty}$ is a $W^{2,p}$-viscosity solution to
	\begin{equation*}
		F_{\infty}\left( D^{2} u_{\infty}, D u_{\infty},u_{\infty},x\right)=f_{\infty} \quad \text { in } \quad \Omega .	
	\end{equation*}
\end{proposition}

In addition to the above definitions and classical results, we will also need some results about the boundary $W^{2,p}$ estimates in the case that  the solution of the homogeneous equation corresponding to $F^{\ast}$ satisfies the $C^{1,1}$ \textit{a priori} estimate.These results will be  involved in the process of our proof of our main results. we  require the following assumptions:
\begin{definition}\label{F^ast}
	The recession function $F^{\ast}$  associated with the operator $F$ is said satisfy  interior and boundary   ${C}^{1,1} $ \textit{a priori}  estimates if the following conditions hold:
	
	Interior case: For every $x_0 \in B_1$ and boundary data $v_0 \in C(\partial B_1)$, there exists a solution $v \in C^{1,1}(B_1) \cap C(\overline{B_1})$ to the Dirichlet problem
	\begin{equation*}
		\left\{
		\begin{array}{ll}
			F^{\ast}(D^2 v, 0, 0, x_0) = 0 & \text{in } B_1, \\
			v = v_0 & \text{on } \partial B_1,
		\end{array}
		\right.
	\end{equation*}
	such that
	\begin{equation*}
		\|v\|_{C^{1,1}(B_{1/2})} \leq C\|v_0\|_{L^\infty(\partial B_1)}.
	\end{equation*}

	Boundary case: For $x_0\in B_1\cap\{x_d=0\}$ and $v_0\in C(\partial B_1^{+})$, $v_0=0$ on $B_1\cap\{x_d=0\}$, there exists a solution $v\in C^{1,1}(B_1^{+})\cap C(\overline{B_1^{+}})$ of
	\begin{equation*}
		\left\{
		\begin{array}{ll}
			F^{\ast}\left(D^{2}v,0,0,x_0\right)=0  &\text{in } B_1^{+},\\
			v=v_0 &\text{on } \partial B_1^{+},
		\end{array}
		\right.
	\end{equation*} 	
	such that
	\begin{equation*}
		\|v\|_{C^{1,1}(\overline{B_{1/2}^{+}}) }\leq C\|v_0\|_{L^{\infty}\left(\partial B_1\right)}.
	\end{equation*}
\end{definition}
Under the previous statements, da Silva and Ricarte \cite{bmo} obtained the following theorem.

\begin{proposition}\label{W2,p}
	Let $\Omega\subset\mathbb{R}^d$, $\partial \Omega\in C^{1,1}$ and $u$ be a $W^{2,p}$-viscosity solution of
	\begin{equation}\label{3.1}
		\left\{
		\begin{array}{ll}
			F(D^2u,Du,u,x)=f(x) & \text{in } \Omega ,\\
			u(x)=\psi(x) &\text{on } \partial \Omega,
		\end{array}
		\right.
	\end{equation}
	where $f\in L^p(\Omega)$ and $\psi\in W^{2,p}(\Omega)$, $p>d$.
	Assume that $F$ satisfies \eqref{F}, the recession function $F^{\ast}$ has ${C}^{1,1} $ \textit{a priori} estimates and satisfies the smallness condition on the oscillation measure in $ L^p$ with $p>d$.
	Then $u\in W^{2,p}(\Omega)$. Moreover, the following estimate hold true
	\begin{equation*}
		\left \| u \right \|_{W^{2,p}(\Omega)}\le C(d,\lambda,\Lambda,p,\theta_0,\alpha_0,\left\|\partial\Omega\right\|_{C^{1,1}})(\left \| u \right \|_{L^\infty(\Omega)}+\left \| \psi \right \|_{W^{2,p}(\Omega)}+\left \| f \right \|_{L^p(\Omega)}).
	\end{equation*}
\end{proposition}
The above \textit{a priori} estimate did not obtain the existence and uniqueness of the $W^{2,p}$-viscosity solution. In our case, the equations obtained by freezing the right-hand term requires the existence of solutions in subsequent proof of our paper. In addition, we  need to strictly adhere to the assumption of convexity rather than \textit{a priori} $C^{1,1}$ regularity. Next, we obtain the following lemma about existence and uniqueness of solution for \eqref{3.1} by using the approximation method.
\begin{proposition}\label{existence}
	Assume that the hypotheses of Proposition \ref{W2,p} hold, but with the $C^{1,1}$  regularity assumption on the recession operator replaced by convexity. Additionally we assume that $F=F(M,p,r,x)$ is non-increasing in $r$. Then there exists a unique $W^{2,p}$-viscosity solution $u$ of \eqref{3.1}.
	Moreover, $u\in W^{2,p}(\Omega)$ and
	\begin{equation*}
		\left\|u\right\|_{W^{2,p}(\Omega)}\le C(\left\|u\right\|_{L^\infty(\Omega)}+\left\|\psi\right\|_{W^{2,p}(\Omega)}+\left\|f\right\|_{L^p(\Omega)}),
	\end{equation*}
	where $C=C(d,p,\lambda,\Lambda,\theta_0,\alpha_0,\gamma,diam(\Omega),\|\partial\Omega\|_{C^{1,1}})>0$ is a universal constant.
\end{proposition}
\begin{proof}
	We could consider solving the "good" problems with a sequence of continuous operators and smoothened right-hand side first. Take a standard mollifier $\phi\in C_0^{\infty}(\mathbb{R}^d)$ with $supp\phi\subset\mathbb{R}^d_{-}$, $\phi\ge 0$ and $\int_{\mathbb{R}^d}\phi(x)dx=1$. Let $\phi_j:=j^{d}\phi(jx)$, $j\in \mathbb{N}$. For $x\in \Omega$, we convolve $\phi_j$ and $F$, then denote
	\begin{equation}\label{F_j}
		F_j(M,p,r,x):=\int_{\mathbb{R}^d}\phi_j(x-y)F(M,p,r,y)dy,
	\end{equation}
	where $F$ extents to 0 outside $\Omega$. 	As well, we could find a sequence of functions $\{f_j\}_{j\in\mathbb{N}}\subset C^{\infty}(\overline{\Omega})\cap L^p(\Omega)$ and
	\begin{equation}\label{f}
		f_j\to f\quad\text{in } L^p(\Omega).
	\end{equation}
	We would like to apply \cite[Proposition 1.11]{winter} to get the existence of a sequence of $C^2$-viscosity solutions $\{u_j\}_{j\in\mathbb{N}}$ for
	\begin{equation}\label{2.5}
		\left\{
		\begin{array}{ll}
			F_j(D^2u_j,Du_j,u_j,x)=f_j(x) & \text{in } B_1^{+}, \\
			u_j(x)=\psi(x) &\text{on } \partial B_1^{+}.
		\end{array}
		\right.
	\end{equation}
	There are some necessary conditions from \cite[Proposition 1.11]{winter} that need to be verified for \eqref{2.5}.
	It is easy to know that $F_j$ satisfies \eqref{F}, $F_j$ is non-increasing in $r$ and $F_j(0,0,0,x)=0$ because the assumption of $F$ can be inherited by $F_j$. Whereupon, we obtain the existence of viscosity solutions for \eqref{2.5}.
	
	Next, after getting the existence of $\{u_j\}_{j\in\mathbb{N}}$, the $W^{2,p}$ estimates of solutions for \eqref{2.5} could be obtained by checking the other conditions in Proposition \ref{W2,p}.
	Firstly, we know that $F^{\ast}_j$ is convex because of the convexity of $F^{\ast}$. Therefore, the famous theory known as the Evans-Krylov theorem in \cite{EK} is applied to solve the following Dirichlet problem
	for fully nonlinear equations
	\begin{equation}\label{fj}
		\left\{
		\begin{array}{ll}
			F^\ast_j(D^2v_j,0,0,x_0)=0 & \text{in } B_1, \\
			v_j(x)=\psi_0(x) &\text{on } \partial B_1,
		\end{array}
		\right.
	\end{equation}
	where $x_0\in B_1$ and $\psi_0\in C(\partial B_1)$.
	Thus, there exists $\overline{\alpha}\in (0,1)$ such that $v_j$ have the $C^{2,\overline{\alpha}}$ \textit{a priori} estimate
	\begin{equation*}
		\|v_j\|_{C^{2,\overline{\alpha}}(B_{1/2})}\le C\|v_j\|_{C^{1,1}(B_1)}
	\end{equation*}
	with the constant $C$ depending only on the ellipticity of $F_j^{\ast}$ and we further get that  $F^{\ast}_j$ has ${C}_{loc}^{1,1} $ estimates.
	
	For $x_0\in B_{1-\epsilon}^{+}$, $0<\epsilon<1$, we have
	\begin{equation}\label{2.7}
		\begin{aligned}
			\left|\beta_{F_j^\ast}(x,x_0)\right|^p & =\left|\sup_{M\in S(d)}\frac{|F_j^{\ast}(M,0,0,x)-F_j^{\ast}(M,0,0,x_0)|}{\left\|M\right\|+1}\right|^p\\
			& = \left|\sup_{M\in S(d)}\lim_{\delta\to 0}\frac{\delta|F_j(\frac{1}{\delta}M,0,0,x)-F_j(\frac{1}{\delta}M,0,0,x_0)|}{\left\|M\right\|+1}\right|^p\\
			& = \left|\sup_{M\in S(d)}\lim_{\delta\to 0}\frac{\delta|\int_{\mathbb{R}^d}\phi_j(y)F(\frac{1}{\delta}M,0,0,x-y)dy-\int_{\mathbb{R}^d}\phi_j(y)F(\frac{1}{\delta}M,0,0,x_0-y)dy|}{\left\|M\right\|+1}\right|^p\\
			& \le \int_{\mathbb{R}^d}\phi_j(y)\left(\sup_{M\in S(d)}\lim_{\delta\to 0}\frac{\delta|F(\frac{1}{\delta}M,0,0,x-y)-F(\frac{1}{\delta}M,0,0,x_0-y)|}{\left\|M\right\|+1}\right)^pdy\\
			& \le \int_{\mathbb{R}^d}\phi_j(y)|\beta_{F^\ast}(x-y,x_0-y)|^pdy,
		\end{aligned}
	\end{equation}
	where the second to last inequality is an application of dominated convergence theorem. Additionally, according to \eqref{beta}, it derives that for $y\in \mathbb{R}^d$,
	\begin{equation}\label{2.8}
		\left(\frac{|B_r(x_0-y)\cap B_1^{+}|}{|B_r(x_0)\cap B_1^{+}|}\fint_{B_r(x_0-y)\cap B_1^{+}}|\beta_{F^\ast}(x,x_0-y)|^p dx\right)^{1/p}\le \theta_0r^{\alpha_0},
	\end{equation}
	where $ \alpha_0 \in(0,1)$, $\theta_{0}>0$  and  $0<r<\epsilon$.
	Afterwards, the assumption about the oscillation of $F_j^{\ast}$ with respect to spatial variable is proved in the light of \eqref{2.7} and \eqref{2.8},
	\begin{equation*}
		\begin{aligned}
			&\left( \frac{1}{|B_r(x_0)\cap B_1^{+}|}\int_{B_r(x_0)\cap B_1^{+}}|\beta_{F_j^\ast}(x,x_0)|^p dx\right) ^{1/p}\\
			&\le\left(\frac{1}{|B_r(x_0)\cap B_1^{+}|}\int_{B_r(x_0)\cap B_1^{+}}\int_{\mathbb{R}^d}\phi_j(y)|\beta_{F^\ast}(x-y,x_0-y)|^pdydx\right)^{1/p}\\
			&\le\left(\int_{\mathbb{R}^d}\phi_j(y)\frac{1}{|B_r(x_0)\cap B_1^{+}|}\int_{B_r(x_0-y)\cap B_1^{+}}|\beta_{F^\ast}(x,x_0-y)|^p dxdy\right) ^{1/p}\\
			&\le\theta_0r^{\alpha_0}\left(\int_{\mathbb{R}^d}\phi_j(y)dy\right) ^{1/p}\\
			&\le\theta_0r^{\alpha_0}.
		\end{aligned}
	\end{equation*}
	Therefore, through the application of Proposition \ref{W2,p} for \eqref{2.5}, ones obtain that
	\begin{equation*}
		\left \| u_j \right \|_{W^{2,p}(\Omega)}\le C(\left \| u_j \right \|_{L^\infty(\Omega)}+\left \| \psi \right \|_{W^{2,p}(\Omega)}+\left \| f_j \right \|_{L^p(\Omega)}),
	\end{equation*}
	where $C=C(d,\lambda,\Lambda,p,\theta_0,\alpha_0,\left\|\partial\Omega\right\|_{C^{1,1}})$.  According to \cite[Proposition 2.13]{caffarelli1995fully}, we know that $u\in \bar{S}(\lambda, \Lambda, \gamma,f)$, thereafter,
	by using Proposition \ref{ABP} and the $L^p$ convergence of $\{f_j\}_{j\in\mathbb{N}}$,
	\begin{equation}\label{3.4}
		\left \| u_j \right \|_{W^{2,p}(\Omega)}\le C(\left \| \psi \right \|_{W^{2,p}(\Omega)}+\left \| f \right \|_{L^p(\Omega)}+\eta),
	\end{equation}
	where $\eta\to 0$ as $j\to\infty$ and $C=C(d,p,\lambda,\Lambda,\theta_0,\alpha_0,\gamma,diam(\Omega),\|\partial\Omega\|_{C^{1,1}})>0$. According to \eqref{3.4}, $\{u_j\}_{j\in\mathbb{N}}$ is uniformly bounded in $W^{2,p}(\Omega)$.
	Because $W^{2,p}(\Omega)$ is a reflexive space, there exists $u\in W^{2,p}(\Omega)$ and a subsequence of functions $\{u_j\}_{j\in\mathbb{N}}$, through a subsequence if necessary, such that $u_j\to u$ weakly in $W^{2,p}(\Omega)$. Then we could get that \eqref{3.4} also holds for $u$,  \begin{equation*}
		\left \| u \right \|_{W^{2,p}(\Omega)}\le C(\left \| \psi \right \|_{W^{2,p}(\Omega)}+\left \| f \right \|_{L^p(\Omega)}) \quad\text{as }j\to\infty.
	\end{equation*}
	In addition, since $p>d$, we have $2p>d$. Then by employing the embedding theorem to subsequence of $\{u_j\}_{j\in\mathbb{N}}$,
	\begin{equation}\label{C}
		u_j\to u \quad\text{in }C(\overline{\Omega}).
	\end{equation}
	Let
	\begin{equation*}
		\begin{aligned}
			g_j(x)&:=F_j(D^2\phi,D\phi,u_j,x)-f_j(x),\\
			g(x)&:=F(D^2\phi,D\phi,u,x)-f(x),
		\end{aligned}
	\end{equation*}
	where $\phi\in W^{2,p}$.
	Then
	\begin{equation*}
		\begin{aligned}
			|g_j(x)-g(x)|&=|F_j(D^2\phi,D\phi,u_j,x)-F_j(D^2\phi,D\phi,u,x)+F_j(D^2\phi,D\phi,u,x)\\&\quad-F(D^2\phi,D\phi,u,x)+f_j(x)-f(x)|\\
			&\le c|u_j-u|+|F_j-F|+|f_j-f|.
		\end{aligned}
	\end{equation*}
	Since \eqref{f}, \eqref{C} and the convergence property of approximations to the identity
	\begin{equation*}
		\left\|F_j(M,p,r,\cdot)-F(M,p,r,\cdot)\right\|_{L^p(\Omega)}\to 0\quad\text{as }j\to\infty,
	\end{equation*}
	one can obtain $\left\|g_j-g \right\|_{L^p}\to 0 $ as $j\to \infty$. Therefore, we use Proposition \ref{stability} to explain the stability of viscosity solutions and get that $u$ is a $W^{2,p}$-viscosity solution of \eqref{3.1}.
	
	In accordance with \cite[Corollary 3.7]{1996}, if $u\in W^{2,p}(\Omega)$ and $u$ is a $W^{2,p}$-viscosity solution to \eqref{3.1}, we have that $u$ is also a $W^{2,p}$-strong solution to \eqref{3.1}. Subsequently, the proof by contradiction is used in the uniqueness of $u$. We assume that $u_1,u_2\in W^{2,p}(\Omega)$ are the $W^{2,p}$-viscosity solutions to \eqref{3.1}, further suppose $u_2$ is a $W^{2,p}$-strong solution. Let $w=u_1-u_2$, it is easy to know that $w$ is a $W^{2,p}$-viscosity solution to
	\begin{equation*}
		\left\{
		\begin{array}{ll}
			G(D^2w,Dw,w,x)=0 & \text{in } B_1, \\
			w(x)=0 &\text{on } \partial B_1,
		\end{array}
		\right.
	\end{equation*}
	where $G(M,p,r,x):=F(M+D^2u_2,p+Du_2,r+u_2,x)-F(D^2u_2,Du_2,u_2,x)$. According to \eqref{F}, $w$ satisfies $\mathcal{M}^{-}(D^2w,\lambda,\Lambda)-\gamma |Dw|\le 0$ in the viscosity sense and then $w\in \bar{S}(\lambda,\Lambda,\gamma,0)$. By Proposition \ref{ABP}, we have $w=0$. Finally, we complete the proof of Proposition \ref{existence}.
\end{proof}

\section{$W^{2,p}$ estimates results for viscosity solutions}\label{section3}

In this section, in order to overcome the nonlocality of nonhomogeneous terms in the Grad-Mercier type equations, we first fix the right-hand term and then prove the existence result of the solution to the frozen problem by the fixed point theorem.
\begin{lemma}\label{delta}
	Let operator $F$ satisfy \eqref{F} and given $\delta>0$,we set
	\begin{equation*}
		G_u^{\delta}(x):=g(\frac{1}{\delta}\int_{0}^{\delta} |\{y\in \Omega:u(y)\ge u(x)-h\}|dh),
	\end{equation*}
	where $f\in L^{p}(\Omega)$ with $p>d$ and $g:[0,|\Omega|]\to \mathbb{R}$ is a continuous function. Suppose that the recession operator $F^\ast$ is convex and satisfies the smallness condition on the oscillation measure in $ L^p$ with $p>d$. Then there exists a $W^{2,p}$-viscosity solution $u_{\delta}$ to
	\begin{equation}\label{ueqn}
		\left\{
		\begin{array}{ll}
			F(D^2u(x),Du(x),u(x),x)=G_u^{\delta}(x)+f(x) & \text{in } \Omega, \\
			u=\psi &\text{on } \partial \Omega.
		\end{array}
		\right.
	\end{equation}
\end{lemma}
\begin{proof}
	We use approximate methods to prove the existence of solutions and proceed by splitting the proof into three steps. Firstly, we investigate the problem of fixing $u$ in $$g(\frac{1}{\delta}\int_{0}^{\delta} |\{y\in \Omega:u(y)\ge u(x)-h\}|dh).$$
	\step We freeze a function $v\in C^{0,1}(\Omega)$ for the right-hand side. That is,
	\begin{equation}\label{veqn}
		\left\{
		\begin{array}{ll}
			F(D^2u(x),Du(x),u(x),x)=G_v^{\delta}(x)+f(x) & \text{in } \Omega, \\
			u=\psi &\text{on } \partial \Omega.
		\end{array}
		\right.
	\end{equation}
	where $G_v^{\delta}(x):=g(\frac{1}{\delta}\int_{0}^{\delta} |\{y\in \Omega:v(y)\ge v(x)-h\}|dh).$
	Furthermore, we could obtain $G_v^{\delta}\in L^p(\Omega)$ because $g:[0,|\Omega|]\to \mathbb{R}$ is a continuous function and $v\in C^{0,1}(\Omega)$.

	By applying Proposition \ref{existence}, we could obtain that there exists a unique $W^{2,p}$-viscosity solution $u$ of \eqref{veqn} and $u$ satisfies the following estimate:
	\begin{equation*}
		\left\|u\right\|_{W^{2,p}(\Omega)}\le C(\left\|u\right\|_{L^\infty(\Omega)}+\left\|\psi\right\|_{W^{2,p}(\Omega)}+\left\|G_v^{\delta}\right\|_{L^p(\Omega)}+\left\|f \right\|_{L^p(\Omega)}),
	\end{equation*}
	where $C=C(d,\lambda,\Lambda,p,\theta_0,\alpha_0,\left\|\partial\Omega\right\|_{C^{1,1}})$.
	
	Next, we plan to apply the fixed point theorem to prove the existence of solutions for \eqref{ueqn}. For this reason, we need to verify whether three conditions of \cite[Theorem 11.3]{GT} are satisfied. We define $T:v\mapsto u$ where $v\in C^{0,1}(\Omega)$ is determined by \eqref{veqn} and $u\in C^{0,1}(\Omega)$ is the solution of \eqref{veqn}.
	\step First we prove that $T$ is a continuous and compact mapping. Because of the existence and uniqueness from previous step, one obtains that $T(v)=u$ is well-defined. We could find a sequence of functions $\{v_j\}_{j\in\mathbb{N}}\subset C^{0,1}(\Omega)$ and $v_j\to v$ in $C^{0,1}(\Omega)$ as $j\to\infty$. Note $u_j:=T(v_j)$, then we claim that $u_j\to T(v)$ in the $C^{0,1}$-norm.
	
	Since $G_{v_j}^{\delta}\in L^p(\Omega)$, we could apply Proposition \ref{W2,p} to derive
	\begin{equation}\label{1}
		\left\|u_j\right\|_{W^{2,p}(\Omega)}\le C(\left\|u_j\right\|_{L^\infty(\Omega)}+\left\|\psi\right\|_{W^{2,p}(\Omega)}+\left\|G_{v_j}^{\delta}\right\|_{L^p(\Omega)}+\left\|f \right\|_{L^p(\Omega)}),
	\end{equation}
	where $C=C(d,\lambda,\Lambda,p,\theta_0,\alpha_0,\left\|\partial\Omega\right\|_{C^{1,1}})$. According to Proposition \ref{ABP}, one obtains for a universal constant $C=C(d,p,\lambda,\Lambda,\theta_0,\alpha_0,\gamma,diam(\Omega),\|\partial\Omega\|_{C^{1,1}})>0$,
	\begin{equation}\label{2}
		\sup_{\Omega}\,u_j\le \sup_{\partial\Omega}\,u_j +C\left(\left\|G_{v_j}^{\delta}\right\|_{L^p(\Omega)}+\left\|f \right\|_{L^p(\Omega)}\right).
	\end{equation}
	Obviously, the following inequality holds true,
	\begin{equation}\label{3}
		\left\|G_{v_j}^{\delta}\right\|_{L^p(\Omega)}\le |\Omega|^{1/p}g(|\Omega|).
	\end{equation}
	By combining \eqref{1}, \eqref{2} and \eqref{3}, we can obtain
	\begin{equation*}
		\left\|u_j\right\|_{W^{2,p}(\Omega)}\le C(\left\|\psi\right\|_{W^{2,p}(\Omega)}+1+\left\|f\right\|_{L^p(\Omega)}).
	\end{equation*}
	Thus, the estimates for $W^{2,p}$-norm of $\{u_j\}_{j\in\mathbb{N}}$ are independent of $\delta$ and $j$.
	
	As a consequence, by using the Rellich-Kondrachov theorem, there exists $u_{\delta}\in C^{0,1}(\Omega)$ such that
	\begin{equation}\label{udelta}
		u_j\to u_{\delta}\quad\text{in the $C^{0,1}$-norm},
	\end{equation}
	through a subsequence if necessary. We know that this embedding is compact and could get the compactness of $T$ easily.
	
	Next, it is significant to prove the continuity of $T$. Suppose that for every problem \eqref{veqn} with $v_j$ in the right-hand side, there exists unique $W^{2,p}$-viscosity solution $u_j\in C(\Omega)$, i.e. $T(v_j)=u_j$. Set
	\begin{equation*}
		\bigtriangleup_j:=\left\|v_j-v\right\|_{L^\infty(\Omega)}
	\end{equation*}
	Since $v_j\to v$ in the $C^{0,1}$-norm, we have
	\begin{equation*}
		\bigtriangleup_j\to 0\quad\text{as }j\to\infty.
	\end{equation*} Fix $x\in\Omega$ and $h\in[0,\delta]$, only consider changes in $j$
	\begin{equation*}
		|v_j\ge v_j(x)-h|\le |v+\bigtriangleup_j\ge v(x)-\bigtriangleup_j-h|,
	\end{equation*}
	then get the right end decreasing to $|v\ge v(x)-h|$ when $j\to\infty$. Similarly,
	\begin{equation*}
		|v_j\ge v_j(x)-h|\ge |v-\bigtriangleup_j\ge v(x)+\bigtriangleup_j-h|,
	\end{equation*}
	however, one only obtains the right end increasing to $|v> v(x)-h|$ when $j\to\infty$. Next, we need to show that $\{v=v(x)-h\}$ is a null set.
	
	For a.e. $\alpha\in V$, we consider $y\in v^{-1}(\alpha)$. According to the corollary of the Rademacher theorem in \cite{evans}, if $v\in C^{0,1}(\Omega)$, the level set $v^{-1}(\alpha)$ is defined away from its critical set $\{\nabla v(y)=0\}$. Therefore,
	\begin{equation*}
		|v^{-1}(v(x)-h)|=|v^{-1}(v(x)-h)\cap \{\nabla v(y)=0\}|.
	\end{equation*}
	Additionally, we get
	\begin{equation*}
		\mathcal{H}^{d-1}(v^{-1}(v(x)-h)\cap \{\nabla v(y)=0\}) =0
	\end{equation*}
	by applying a corollary of the coarea formula, where $\mathcal{H}^{d-1}$ stands for the $(d-1)$-dim Hausdorff measure.
	
	Then, we get
	\begin{equation*}
		|v=v(x)-h|=|v^{-1}(v(x)-h)|=0\quad \text{a.e. } h\in[0,\delta],
	\end{equation*}
	and further obtain $|v>v(x)-h|=|v\ge v(x)-h|$ for a.e. $h\in [0,\delta].$ For every $x\in\Omega$,
	\begin{equation*}
		|v_j\ge v_j(x)-h|\to |v\ge v(x)-h| \quad \text{a.e. } h\in[0,\delta].
	\end{equation*}
	Therefore, as $j\to\infty$
	\begin{equation*}
		g\left(\frac{1}{\delta}\int_{0}^{\delta} |\{y\in \Omega:v_j(y)\ge v_j(x)-h\}|dh\right)\to g\left(\frac{1}{\delta}\int_{0}^{\delta} |\{y\in \Omega:v(y)\ge v(x)-h\}|dh\right)
	\end{equation*}
	through applying the dominated convergence theorem and the continuity of $g$. Thus,
	\begin{equation}\label{3.13}
		G_{v_j}^{\delta}(x)\to G_v^{\delta}(x) \quad\text{in the $L^p$-norm as }j\to\infty
	\end{equation}
	by using the dominated convergence theorem again.
	\eqref{udelta} and \eqref{3.13} satisty the conditions in Proposition \ref{stability}, thus we get $T(v)=u_{\delta}$ by applying the stability theorem of viscosity solutions. Then
	\begin{equation*}
		T(v_j)=u_j\to u_{\delta}=T(v) \quad\text{ in the $C^{0,1}$-norm }\text{as }v_j\to v.
	\end{equation*}
	So we obtain that $T$ is continuous.

	Finally, one is just going to prove the boundness of $v\in C^{0,1}(\Omega)$ satisfying $v=\gamma\, Tv$ with $\gamma\in [0,1]$.
	\step Suppose that $S:=\{v\in C^{0,1}(\Omega):\exists\,\gamma\in[0,1]\text{ s.t. } v=\gamma\, T(v)\}$ and intend to prove that $S$ is bounded. We argue by contradiction and suppose the statement is false. In this case, there exist a sequence of nonzero functions $\{v_n\}_{n\in\mathbb{N}}\subset S$ and a sequence of nonzero numbers $\{\gamma_n\}_{n\in\mathbb{N}}$ such that $v_n=\gamma_n\, T(v_n)$ and
	\begin{equation}\label{vn}
		\left\|v_n\right\|_{C^{0,1}(\Omega)}\to \infty\quad\text{as }n\to\infty.
	\end{equation}  It is easy to know that $0\in S$ when $\gamma=0$ while for every $0\ne v\in S$, there exists $\gamma\ne 0$ which matches $v$. Note that $w_n:=v_n/\gamma_n\in C^{0,1}(\Omega)$ for every $n\in \mathbb{N}$, so $w_n\in W^{2,p}(\Omega)$ and furthermore,
	\begin{equation*}
		\left\|v_n\right\|_{C^{0,1}(\Omega)}\le \left\|w_n\right\|_{C^{0,1}(\Omega)}\le C\left\|w_n\right\|_{W^{2,p}(\Omega)}\le \tilde{C},
	\end{equation*}
	where $\tilde{C}>0$ is a universe constant and which is a contradiction with \eqref{vn}. So $S$ is a bounded set.
	
	We have proven that all three hypotheses of the fixed point theorem hold true, so there exists fixed point $u_\delta\in C^{0,1}(\Omega)$ for mapping $T$, i.e. $u_\delta=T(u_\delta)$. The existence of solutions to the Dirichlet problem
	\eqref{ueqn} is obtained and the proof of Lemma \ref{delta}	is completed from this.
\end{proof}
By using approximation methods for the solution $u_{\delta}$ which is from Lemma \ref{delta}, we prove the existence and $W^{2,p}$ regularity for viscosity solutions to Dirichlet problem \eqref{eqn} for the fully nonlinear elliptic equation of Grad-Mercier type.
\begin{proof}[Proof of Theorem \ref{main}]
	For every $\delta>0$, there exists a solution $u_\delta\in W^{2,p}(\Omega)$ with uniformly bounded $W^{2,p}$-norm. Find a sequence of functions $\{u_\delta\}$, and a function $u\in C^{0,1}(\Omega)$ such that
	\begin{equation}\label{c0,1}
		u_\delta\to u\quad\text{in the $C^{0,1}$-norm as $\delta\to 0$}.
	\end{equation} Next we are going to apply Proposition \ref{stability} again, only need to prove the convergence of right-hand side.
	
	Let
	\begin{equation*}
		\begin{aligned}
			G_u(x)&:=g(|u\ge u(x)|),\\
			G_{u_\delta}^\delta(x)&:=g\left(\frac{1}{\delta}\int_0^\delta|u_\delta\ge u_\delta(x)-h|dh\right).
		\end{aligned}
	\end{equation*}
	By using the triangle inequality, we have
	\begin{equation*}
		\left\|G_u-G_{u_\delta}^\delta\right\|_{L^p(\Omega)}\le \left\|G_u-G_{u}^\delta\right\|_{L^p(\Omega)}+\left\|G_u^\delta-G_{u_\delta}^\delta\right\|_{L^p(\Omega)}.
	\end{equation*}
	According to \eqref{3.13}, one obtains that
	\begin{equation}\label{conv1}
		\left\|G_u^\delta-G_{u_\delta}^\delta\right\|_{L^p(\Omega)}\to 0\quad\text{as }\delta\to 0.
	\end{equation} Meanwhile, for $h\in[0,\delta]$,
	\begin{equation*}
		\begin{aligned}
			|u\ge u(x)|&\le|u\ge u(x)-h|\\
			&=|u\ge u(x)|+|u(x)>u\ge u(x)-h|\\
			&\le|u\ge u(x)|+|u(x)>u\ge u(x)-\delta|.
		\end{aligned}
	\end{equation*}
	From this,
	\begin{equation*}
		\frac{1}{\delta}\int_0^\delta|u\ge u(x)-h|dh\to |u\ge u(x)| \quad\text{as }\delta\to 0,
	\end{equation*}
	and since $g:[0,|\Omega|]\to \mathbb{R}$ is a continuous function,
	\begin{equation*}
		g\left(\frac{1}{\delta}\int_0^\delta|u\ge u(x)-h|dh\right)\to g\left(|u\ge u(x)|\right) \quad\text{as }\delta\to 0.
	\end{equation*}
	Subsequently, we use the dominated convergence theorem to get
	\begin{equation}\label{conv2}
		\left\|G_u-G_{u}^\delta\right\|_{L^p(\Omega)}\to 0\quad\text{as }\delta\to 0.
	\end{equation}
	
	In light of \eqref{conv1} and \eqref{conv2}, the convergence for $G_{u_{\delta}}^\delta$ to $G_u$ in the $L^p$-norm is acquired, which is applied in Proposition \ref{stability} to get that $u$ is a $W^{2,p}$-viscosity solution to \eqref{eqn}. In addition, on the basis of \eqref{1}, let $j\to\infty$,
	\begin{equation}\label{udeltaw2,p}
		\left\|u_{\delta}\right\|_{W^{2,p}(\Omega)}\le C(\left\|u_{\delta}\right\|_{L^\infty(\Omega)}+\left\|\psi\right\|_{W^{2,p}(\Omega)}+\left\|G_{u_{\delta}}^\delta\right\|_{L^p(\Omega)}+\left\|f\right\|_{L^p(\Omega)}),
	\end{equation}
	where $C=C(d,p,\lambda,\Lambda,\theta_0,\alpha_0,\gamma,diam(\Omega),\|\partial\Omega\|_{C^{1,1}})>0$. Since \eqref{c0,1} and $G^{\delta}_{u_\delta}\to G_u$ in the $L^p$-norm as $\delta\to 0$,
	\begin{equation*}
		\left\|u\right\|_{W^{2,p}(\Omega)}\le C\left(\|u\|_{L^{\infty}\left(\Omega\right)}+\|\psi\|_{W^{2,p}\left(\Omega\right)}+\|G_u\|_{L^p(\Omega)}+\|f\|_{L^p(\Omega)}\right).
	\end{equation*}
	Finally, we complete the proof of Theorem \ref{main}.
\end{proof}

\section{Further $p\operatorname{-BMO}$ results}\label{section4}
In this section, we obtain a global $p\operatorname{-BMO}$ estimate for the second derivatives of solutions refining the borderline case of Calder\'{o}n-Zygmund estimates by making use of the similar approximation strategy. First of all, under the assumption of asymptotic convexity, we derive the $p\operatorname{-BMO}$ \textit{a priori} estimate for solutions to the fully nonlinear equation without inhomogeneous term $g$. Similar to how Proposition \ref{W2,p} is used in Theorem \ref{main}, this $\operatorname{BMO}$ type estimate serves a key role in the subsequent proof of Theorem \ref{pbmo}.

\begin{proposition}\label{pe}
	Let $\Omega\subset \mathbb{R}^d$, $\partial\Omega\in C^{1,1}$ and $u$ be a $W^{2,p}$-viscosity solution of \eqref{3.1}, where $f\in p\operatorname{-BMO}(\Omega)$ and $D^2\psi\in p\operatorname{-BMO}(\Omega)$. Assume that $F$ satisfies \eqref{F}. Suppose that the recession operator $F^\ast$ is convex and satisfies the smallness condition on the oscillation measure in $ L^p$ with $p>d$. Then $D^2u\in p\operatorname{-BMO}(\Omega)$ and
	\begin{equation*}
		\|D^2u\|_{p\operatorname{-BMO}(\Omega)}\le C(\left \| u \right \|_{L^\infty(\Omega)}+\|D^2\psi\|_{p\operatorname{-BMO}(\Omega)}+\|f\|_{p\operatorname{-BMO}(\Omega)}),
	\end{equation*}
	where $C=C(d,p,\lambda,\Lambda,\theta_0,\alpha_0,\gamma,\omega,diam(\Omega),\|\partial\Omega\|_{C^{1,1}})>0$.
\end{proposition}

\begin{proof}
	Firstly, we would get rid of the impact of lower order terms. According to \cite[Theorem 3.6]{1996}, we know that $u$ is twice differential a.e. in $\Omega$ and the pointwise derivatives satisfy the equation \eqref{3.1} almost everywhere. Define $\tilde{f}:=F(D^2u,0,0,x)$, then
	\begin{equation*}
		\begin{aligned}
			|\tilde{f}(x)|&\le |F(D^2u,0,0,x)-F(D^2u,Du,u,x)|+|F(D^2u,Du,u,x)|\\
			&\le \gamma|Du|+\omega|u|+f(x) \quad\text{ a.e. } x\in \Omega,
		\end{aligned}
	\end{equation*}
	where $\gamma$ and $\omega$ are constants from \eqref{F}. We know that $u\in W^{2,p}(\Omega)$ which is derived from Proposition \ref{W2,p} and according to the Poincar\'{e} inequality,
	\begin{equation}\label{u2,p}
		\begin{aligned}
			[Du]_{p\operatorname{-BMO}(\Omega)}&\le C\|D^2u\|_{L^p(\Omega)}\\
			&\le C\left(\|u\|_{L^{\infty}\left(\Omega\right)}+\|\psi\|_{W^{2,p}\left(\Omega\right)}+\|f\|_{L^p(\Omega)}\right).
		\end{aligned}
	\end{equation} Since $f\in p\operatorname{-BMO}(\Omega)$ and \eqref{u2,p}, one obtains $\tilde{f}\in p\operatorname{-BMO}(\Omega)$ and
	\begin{equation}\label{ftilde}
		\begin{aligned}
			\|\tilde{f}\|_{p\operatorname{-BMO}(\Omega)}
			&\le C\left(\|Du\|_{p\operatorname{-BMO}(\Omega)}+\|u\|_{p\operatorname{-BMO}(\Omega)}+\|f\|_{p\operatorname{-BMO}(\Omega)}\right)\\
			&\le C\left(\|u\|_{L^{\infty}\left(\Omega\right)}+\|\psi\|_{W^{2,p}\left(\Omega\right)}+\|f\|_{p\operatorname{-BMO}(\Omega)}\right)
		\end{aligned}
	\end{equation}
	where the universal constant $C=C(d,p,\Omega)$. Then it follows that $u$ satisfies $\tilde{F}(D^2u,x):=F(D^2u,0,0,x)=\tilde{f}$ in $\Omega$ in the viscosity sense by using \cite[Corollary 1.6]{du1}. Thus, by using flatting boundary theory we only study the following equation
	\begin{equation*}
		\left\{
		\begin{array}{ll}
			\tilde{F}(D^2u,x)=\tilde{f} & \text{in } B_1^{+}, \\
			u=\psi &\text{on } \partial B_1^{+}\cap \{x_d=0\}.
		\end{array}
		\right.
	\end{equation*}
	Let $w=u-\psi$ and it is easy to know that $w$ is a $W^{2,p}$-viscosity solution to
	\begin{equation}\label{w}
		\left\{
		\begin{array}{ll}
			G(D^2w,x)=g &\text{in }B_1^{+}, \\
			w=0 &\text{on } \partial B_1^{+}\cap \{x_d=0\}.
		\end{array}
		\right.
	\end{equation}
	where $G(M,x):=\tilde{F}(M+D^2\psi,x)-\tilde{F}(D^2\psi,x)$ and $g(x):=\tilde{f}(x)-\tilde{F}(D^2\psi,x)$. Since \eqref{ftilde} and
	\begin{equation}
		\tilde{F}(D^2\psi,x)-\tilde{F}(0,x)\le C(d,\lambda,\Lambda,p,\Omega)\,\|D^2\psi\|_{p\operatorname{-BMO}(\Omega)}
	\end{equation}we have $g\in p\operatorname{-BMO}(B_1^{+})$ and
	\begin{equation}\label{gp}
		\|g\|_{p\operatorname{-BMO}(B_1^{+})}\le C(\|u\|_{L^{\infty}\left(\Omega\right)}+\|\psi\|_{W^{2,p}\left(B_1^{+}\right)}+\|f\|_{p\operatorname{-BMO}(B_1^{+})}+\|D^2\psi\|_{p\operatorname{-BMO}(\Omega)}).
	\end{equation}
	In order to prove $D^2u\in p\operatorname{-BMO}(\Omega)$, we first claim that $D^2w\in p\operatorname{-BMO}(\Omega)$.
	
	Here is an approximation problem for \eqref{w} and $w_\mu(x):=\mu\, w(x)$ is the $W^{2,p}$ viscosity solution to
	\begin{equation}\label{vx}
		\left\{
		\begin{array}{ll}
			G_{\mu}(D^2w_\mu,x)=\mu\, g & \text{in } B_1^{+}, \\
			w_\mu=0 &\text{on } \partial B_1^{+}\cap \{x_d=0\},
		\end{array}
		\right.
	\end{equation}
	where $G_\mu(M,x)=\mu\, G(\frac{1}{\mu} M,x)$ with $\mu\in (0,1)$.
	We intend to apply \cite[Corollary 4.5]{bmo} to \eqref{vx} which requires that $w_\mu$ is a normalized viscosity solution and  for $0<\epsilon_0\ll 1$ and sufficiently small parameter $\mu$ satisfies
	\begin{equation}\label{ep}
		\max(\mu,[\mu\, g]_{p\operatorname{-BMO}(B_1^{+})})\le \epsilon_0.
	\end{equation}
	Therefore, consider
	\begin{equation}\label{v}
		v:=\frac{\epsilon_0\,w_\mu}{\epsilon_0\|w_\mu\|_{L^\infty(B_1^{+})}+\|\mu\, g\|_{p\operatorname{-BMO}(B_1^{+})}},
	\end{equation}
	and note that $v$ is the viscosity solution to
	\begin{equation}
		\left\{
		\begin{array}{ll}
			H_\mu(D^2v,x):=\frac{1}{K}G_{\mu}(KD^2v,x)=\frac{1}{K}\mu\, g & \text{in } B_1^{+}, \\
			v=0 &\text{on } \partial B_1^{+}\cap \{x_d=0\},
		\end{array}
		\right.
	\end{equation}
	with $K=\epsilon_0^{-1}\left(\epsilon_0\|w_\mu\|_{L^\infty(B_1^{+})}+\|\mu\, g\|_{p\operatorname{-BMO}(B_1^{+})}\right)$.
	
	If we want to get $D^2w\in p\operatorname{-BMO}(\Omega)$, just need to obatain the $p\operatorname{-BMO}$ estimate of $D^2v$. Since
	\begin{equation*}
		|D^2v(x)-(D^2v)_{x_0,\rho}|\le |D^2v(x)-M|+|M-(D^2v)_{x_0,\rho}|\le |D^2v(x)-M|+\fint_{B_{\rho}(x_0)\cap\Omega}|D^2v(x)-M|dx,
	\end{equation*}
	to prove that for $B\subset\Omega$ $$\sup_{B}\left(\fint_{B}|D^2v(x)-(D^2v)_{x_0,\rho}|^{p} d x\right)^{\frac{1}{p}} < \infty,$$ it suffices to show that there exists $M\in \mathcal{S}(d)$ such that
	\begin{equation}\label{sup}
		\sup_{B_\rho\subset\Omega}\left(\fint_{B_\rho(x_0)\cap\Omega}|D^2v(x)-M|^{p} d x\right)^{\frac{1}{p}} < \infty.
	\end{equation}
	Now, we claim that for $n\in \mathbb{N}$ and some $\rho\in (0,1)$, there exists a sequence of quadratic paraboloids $P_n(x)=\frac{1}{2}x^TA_nx+b_n\cdot x+c_n$ satisfying
	\begin{equation}\label{a1}
		H^{\ast}(D^2P_n,x)=H^{\ast}(A_n,x)=0,
	\end{equation}
	\begin{equation}\label{a2}
		\rho^{2(n-1)}|A_n-A_{n-1}|+\rho^{n-1}|b_n-b_{n-1}|+|c_n-c_{n-1}|\le C\rho^{2(n-1)}
	\end{equation}
	and
	\begin{equation}\label{a3}
		\sup_{B_{\rho^n}^{+}}|v-P_n(x)|\le \rho^{2n},
	\end{equation}
	where $H^{\ast}(M,x)=\lim\limits_{\mu \to 0}H_\mu(M,x)=\lim\limits_{\mu \to 0}\mu H(\frac{1}{\mu}M,x)$.
	
	Induction method is used in the proof to verify this assertion. To begin with, let $P_0(x)=P_{-1}(x)=\frac{1}{2}x^TA_0x$ with $A_0\in\mathcal{S}(d)$ satisfying $H^{\ast}(A_0,x)=0$. For $j\in\mathbb{N}$, construct a sequence of functions
	\begin{equation}\label{vj}
		v_j(x):=\frac{(v-P_j)(\rho^jx)}{\rho^{2j}},
	\end{equation}
	and then $v_j$ is the viscosity solution to
	\begin{equation}\label{hmuj}
		\left\{
		\begin{array}{ll}
			H_{\mu,j}(D^2v_j(x),x)=g_j & \text{in } B_\rho^{+}, \\
			v_j=0 &\text{on } \partial B_\rho^{+}\cap \{x_d=0\},
		\end{array}
		\right.
	\end{equation}	
	where $H_{\mu,j}(D^2v_j(x),x)=H_{\mu}(D^2v_j(x)+A_j,\rho^jx)$ and $g_j(x)=\mu\, g(\rho^jx)$. According to Proposition \ref{existence},  the existence of $v_j$ to \eqref{hmuj} is obtained. Moreover, $v_j\in W^{2,p}(B_1^{+})$ and
	\begin{equation}\label{vjw2p}
		\left\|v_j\right\|_{W^{2,p}(B_1^{+})}\le C\left\|v_j\right\|_{L^\infty(B_1^{+})}.
	\end{equation}
	Since \eqref{ep}, we have
	\begin{equation*}
		\begin{aligned}
			[g_j]_{p\operatorname{-BMO}(B_1^{+})}&=\sup_{0<\rho_0\le 1}\left(\fint_{B_{\rho_0}^{+}\cap\Omega}|g_j(x)-(g_j)_{\rho_0}|^{p} d x\right)^{\frac{1}{p}}\\
			&=\sup_{0<\rho_0\le 1}\mu\left(\fint_{B_{\rho_0\rho}^{+}\cap\Omega}|g(x)-(g)_{\rho_0\rho}|^{p} d x\right)^{\frac{1}{p}}\\
			&\le [\mu\, g]_{p\operatorname{-BMO}(B_1^{+})}\le \epsilon_0.
		\end{aligned}
	\end{equation*}
	Let $\mu\to 0$ and suppose that $h$ is a viscosity solution to
	\begin{equation}\label{hj}
		\left\{
		\begin{array}{ll}
			H_{j}^{\ast}(D^2h,x)=0 & \text{in } B_\rho^{+}, \\
			h=0 &\text{on } \partial B_\rho^{+}\cap \{x_d=0\}.
		\end{array}
		\right.
	\end{equation}
	By assuming that $H_{j}^{\ast}(0,x)=H^{\ast}(A_j,x)=0$ for every $j\le n$ and $H_j^{\ast}$ is convex because ${F}^{\ast}$ is convex, in accordance with \cite[Theorem 9.7]{caffarelli1995fully}, we have the existence and uniqueness of $h$ to the Dirichlet problem \eqref{hj} and furthermore, $h\in C^{2,\alpha}(B_\rho^{+})$ with $\alpha\in(0,1)$. On the basis of \cite[Lemma 4.3]{bmo}, we have
	\begin{equation}\label{vjh}
		\|v_j-h\|_{L^\infty(B_\rho^{+})}\le\delta\quad\text{with }\delta>0
	\end{equation}
	Subsequently, we have proved the requirements in \cite[Corollary 4.5]{bmo} and by using this corollary, there exists a quadratic polynomial $\tilde{P}(x)$ and $\rho\in(0,1)$ such that
	\begin{equation}\label{r2}
		\sup_{B_{\rho}^{+}}|v_j-\tilde{P}(x)|\le \rho^{2}.
	\end{equation}
	Substituting \eqref{vj} into \eqref{r2}, we obtain
	\begin{equation}
		\sup_{B_{\rho^{j+1}}^{+}}|(v-P_j)(x)-\tilde{P}(\frac{x}{\rho^j})|\le \rho^{2+2j}.
	\end{equation}
	Denote $$P_{n+1}(x)=P_n(x)+\tilde{P}(\frac{x}{\rho^n})\rho^{2n}.$$ It is easy to check that \eqref{a1}, \eqref{a2} and \eqref{a3} are true for $P_{n+1}(x)$. Therefore, claim above is established.
	
	For $\forall\, 0<r<1$, there exist $k\in \mathbb{N}$ and a constant $C$ such that $0<\rho^{k+1}<r\le \rho^k$ and by applying \eqref{vjw2p} and \eqref{vjh},
	\begin{equation*}
		\begin{aligned}
			\left(\fint_{B_r^{+}\cap\Omega}|D^2v(x)-A_k|^{p} d x\right)^{\frac{1}{p}}&\le \left(\frac{1}{|B_{\rho^{k+1}}^{+}\cap\Omega|}\int_{B_{\rho^k}^{+}\cap\Omega}|D^2v(x)-A_k|^{p} d x\right)^{\frac{1}{p}}\\
			&\le C\left(\fint_{B_1^{+}\cap\Omega}|D^2v_k|^{p} dx\right)^{\frac{1}{p}}\le C.
		\end{aligned}
	\end{equation*}
	Let $M=A_k$ in \eqref{sup} and by covering theorem, there exists a universal constant $$C=C(d,p,\lambda,\Lambda,\theta_0,\alpha_0,\gamma,\omega,diam(\Omega),\|\partial\Omega\|_{C^{1,1}})>0$$ such that
	\begin{equation}\label{d2v}
		\|D^2v\|_{p\operatorname{-BMO}(\Omega)}\le C.
	\end{equation}
	Combining \eqref{gp}, \eqref{v} and \eqref{d2v}, we have
	\begin{equation*}
		\begin{aligned}
			\|D^2u\|_{p\operatorname{-BMO}(\Omega)}&\le \frac{1}{\mu}\|D^2w_\mu\|_{p\operatorname{-BMO}(\Omega)}+\|D^2\psi\|_{p\operatorname{-BMO}(\Omega)}\\
			&\le \frac{1}{\mu\epsilon_0}\|D^2v\|_{p\operatorname{-BMO}(\Omega)}\left(\epsilon_0\|w_\mu\|_{L^\infty(\Omega)}+\|\mu\, g\|_{p\operatorname{-BMO}(\Omega)}\right)+\|D^2\psi\|_{p\operatorname{-BMO}(\Omega)}\\
			&\le C\left(\|w\|_{L^\infty(\Omega)}+\| g\|_{p\operatorname{-BMO}(\Omega)}+\|D^2\psi\|_{p\operatorname{-BMO}(\Omega)}\right)\\
			&\le C\left(\|u\|_{L^\infty(\Omega)}+\| f\|_{p\operatorname{-BMO}(\Omega)}+\| D^2\psi\|_{p\operatorname{-BMO}(\Omega)}\right),
		\end{aligned}
	\end{equation*}
	which finishes the proof of Proposition \ref{pe}.
\end{proof}
On the basis of Proposition \ref{pe}, we obtain the following corollary. Since the proof closely follows that of Lemma \ref{existence}, we omit the details here.
\begin{corollary}\label{cor}
	Assume that the hypotheses of Proposition \ref{pe} hold. Then there exists a unique $W^{2,p}$ viscosity solution $u$ to problem \eqref{3.1} with $D^2\psi\in p\operatorname{-BMO}(\Omega)$ and
	\begin{equation*}
		\|D^2u\|_{p\operatorname{-BMO}(\Omega)}\le C(\left \| u \right \|_{L^\infty(\Omega)}+\|D^2\psi\|_{p\operatorname{-BMO}(\Omega)}+\|f\|_{p\operatorname{-BMO}(\Omega)}),
	\end{equation*}
	where $C=C(d,p,\lambda,\Lambda,\theta_0,\alpha_0,\gamma,\omega,diam(\Omega),\|\partial\Omega\|_{C^{1,1}})>0$.
\end{corollary}
Finally, we provide a brief proof for Theorem \ref{pbmo} as follows.
\begin{proof}[Proof of Theorem \ref{pbmo}]
	Because of Theorem \ref{main}, we already have the existence of $W^{2,p}$ viscosity solution for \eqref{eqn}. According to Corollary \ref{cor}, we could get a sequnce of functions $\{u_\delta\}_{\delta>0}$ satisfying $\{D^2u_{\delta}\}_{\delta>0}\subset p\operatorname{-BMO}(\Omega)$ as well as the equation \eqref{ueqn} with frozen right term. It follows that one proceeds similar to the proof of Theorem \ref{main} and replaces \eqref{udeltaw2,p} with
	
	\begin{equation*}
		\|D^2u_\delta\|_{p\operatorname{-BMO}(\Omega)}\le C(\left \| u_\delta \right \|_{L^\infty(\Omega)}+\|D^2\psi\|_{p\operatorname{-BMO}(\Omega)}+\left\|G_{u_\delta}^\delta\right\|_{p\operatorname{-BMO}(\Omega)}+\|f\|_{p\operatorname{-BMO}(\Omega)})
	\end{equation*} at the end of this proof, where $C=C(d,p,\lambda,\Lambda,\theta_0,\alpha_0,\gamma,\omega,diam(\Omega),\|\partial\Omega\|_{C^{1,1}})>0$. In the following, as $\delta\to 0$, we have
	\begin{equation*}
		\|D^2u\|_{p\operatorname{-BMO}(\Omega)}\le C(\left \| u \right \|_{L^\infty(\Omega)}+\|D^2\psi\|_{p\operatorname{-BMO}(\Omega)}+\|g(|u\ge u(x)|)\|_{p\operatorname{-BMO}(\Omega)}+\|f\|_{p\operatorname{-BMO}(\Omega)}).
	\end{equation*}Finally, we obtain that solutions satisfy $ p\operatorname{-BMO}$ estimates for the Grad-Mercier type equation \eqref{eqn}.
\end{proof}

\section*{Data availability}
 Data will be made available on request.

\section*{Acknowledgments}
 The authors are supported by the National Natural Science Foundation of China (NNSF  Grant No.12071229 and No.12101452), Young Scientific and Technological Talents (Level Three) in Tianjin.

\end{document}